\documentclass[reqno,10pt,centertags]{amsart}
\usepackage{amsmath,amsthm,amscd,amssymb,latexsym,upref,mathrsfs}

\usepackage{hyperref}
\newcommand*{\mailto}[1]{\href{mailto:#1}{\nolinkurl{#1}}}

\newcommand{\bbN}{{\mathbb{N}}}
\newcommand{\bbR}{{\mathbb{R}}}

\newcommand{\bbZ}{{\mathbb{Z}}}
\newcommand{\bbC}{{\mathbb{C}}}

\newcommand{\cB}{{\mathcal B}}

\newcommand{\cD}{{\mathcal D}}

\newcommand{\cH}{{\mathcal H}}

\newcommand{\cJ}{{\mathcal J}}

\newcommand{\cM}{{\mathcal M}}

\newcommand{\cW}{{\mathcal W}}
\newcommand{\cX}{{\mathcal X}}

\newcommand{\no}{\notag}
\newcommand{\lb}{\label}

\newcommand{\ol}{\overline}

\newcommand{\wti}{\widetilde}

\newcommand{\f}{\frac}
\newcommand{\bi}{\bibitem}
\newcommand{\hatt}{\widehat}

\renewcommand{\Re}{\mathop\mathrm{Re}}

\renewcommand{\le}{\leqslant}

\DeclareMathOperator{\dom}{dom}
\DeclareMathOperator{\ran}{ran}

\DeclareMathOperator*{\nlim}{n-lim}
\DeclareMathOperator*{\slim}{s-lim}

\newcommand{\Om}{\Omega}

\allowdisplaybreaks \numberwithin{equation}{section}

\newtheorem{theorem}{Theorem}[section]

\newtheorem{lemma}[theorem]{Lemma}

\newtheorem{hypothesis}[theorem]{Hypothesis}
\newtheorem{example}[theorem]{Example}
\theoremstyle{remark}
\newtheorem{remark}[theorem]{Remark}

\begin{document}

\title[The Krein--von Neumann Extension]{The Krein--von Neumann Extension and its
Connection to an Abstract Buckling Problem}

\author[M.\ S.\ Ashbaugh]{Mark S.\ Ashbaugh}
\address{Department of Mathematics,
University of Missouri, Columbia, MO 65211, USA}
\email{\mailto{ashbaughm@missouri.edu}}
\urladdr{\url{http://www.math.missouri.edu/personnel/faculty/ashbaughm.html}}

\author[F.\ Gesztesy]{Fritz Gesztesy}
\address{Department of Mathematics,
University of Missouri, Columbia, MO 65211, USA}
\email{\mailto{gesztesyf@missouri.edu}}
\urladdr{\url{http://www.math.missouri.edu/personnel/faculty/gesztesyf.html}}

\author[M.\ Mitrea]{Marius Mitrea}
\address{Department of Mathematics, University of
Missouri, Columbia, MO 65211, USA}
\email{\mailto{mitream@missouri.edu}}
\urladdr{\url{http://www.math.missouri.edu/personnel/faculty/mitream.html}} 

\author[R.\ Shterenberg]{Roman Shterenberg}
\address{Department of Mathematics, University of Alabama at Birmingham, 
Birmingham, AL 35294, USA}
\email{\mailto{shterenb@math.uab.edu}}

\author[G. Teschl]{Gerald Teschl}
\address{Faculty of Mathematics\\ University of Vienna\\ 
Nordbergstrasse 15\\ 1090 Wien\\ Austria\\ and International Erwin Schr\"odinger
Institute for Mathematical Physics\\ Boltzmanngasse 9\\ 1090 Wien\\ Austria} 
\email{\mailto{Gerald.Teschl@univie.ac.at}}
\urladdr{\url{http://www.mat.univie.ac.at/~gerald/}}

\thanks{Based upon work partially supported by the US National Science
Foundation under Grant Nos.\ DMS-0400639 and
FRG-0456306 and the Austrian Science Fund (FWF) under Grant No.\ Y330.}
\dedicatory{Dedicated to the memory of Erhard Schmidt (1876--1959).}
\thanks{Math. Nachr. {\bf 283:2}, 165--179 (2010).}
\date{\today}

\begin{abstract}
We prove the unitary equivalence of the inverse of the Krein--von Neumann extension 
(on the orthogonal complement of its kernel) of a densely defined, closed, strictly positive
operator, $S\geq \varepsilon I_{\mathcal{H}}$ for some $\varepsilon >0$ in a Hilbert space $\mathcal{H}$ to an abstract buckling problem operator. 

In the concrete case where $S=\overline{-\Delta|_{C_0^\infty(\Omega)}}$ in 
$L^2(\Omega; d^n x)$ for $\Omega\subset\mathbb{R}^n$ an open, bounded (and sufficiently regular) domain, this recovers, as a particular case of a general result due to G. Grubb, that the eigenvalue problem for the Krein Laplacian $S_K$ (i.e., the Krein--von Neumann extension of $S$), 
\[
S_K v = \lambda v, \quad \lambda \neq 0, 
\]
is in one-to-one correspondence with the problem of {\em the buckling of a clamped plate},
\[
(-\Delta)^2u=\lambda (-\Delta) u \, \text{ in } \, \Omega, \quad \lambda \neq 0, 
\quad u\in H_0^2(\Omega),  
\]
where $u$ and $v$ are related via the pair of formulas 
\[
u = S_F^{-1} (-\Delta) v, \quad v = \lambda^{-1}(-\Delta) u,
\]
with $S_F$ the Friedrichs extension of $S$.

This establishes the Krein extension as a natural object in elasticity theory
(in analogy to the Friedrichs extension, which found natural applications in quantum
mechanics, elasticity, etc.). 
\end{abstract}

\maketitle


\section{Introduction}
\label{s1}

Suppose that $S$ is a densely defined, symmetric, closed operator with nonzero deficiency indices in a separable complex Hilbert space $\cH$ that satisfies 
\begin{equation}
S\geq \varepsilon I_{\cH} \, \text{ for some $\varepsilon >0$,}    \lb{1.1}
\end{equation}
and denote by $S_K$ and $S_F$ the Krein--von Neumann and Friedrichs extensions of 
$S$, respectively (with $I_{\cH}$ the identity operator in $\cH$).

Then an abstract version of Proposition\ 1 in Grubb \cite{Gr83}, describing an  intimate connection between the nonzero eigenvalues of the Krein--von Neumann extension of an appropriate minimal elliptic differential operator of order $2m$, 
$m\in\bbN$, and nonzero eigenvalues of a suitable higher-order buckling problem (cf.\ Example \ref{e3.6}), to be proved in Lemma \ref{l3.3}, can be summarized as  follows:
\begin{align}
& \text{There exists $0 \neq v \in \dom(S_K)$ satisfying } \,  S_K v = \lambda v, 
\quad \lambda \neq 0,   \lb{1.1a} \\
& \text{if and only if}  \no \\
& \text{there exists a $0 \neq u \in \dom(S^* S)$ such that } \, 
S^* S u = \lambda S u,   \lb{1.1b} 
\end{align}
and the solutions $v$ of \eqref{1.1a} are in one-to-one correspondence with the solutions $u$ of \eqref{1.1b} given by the pair of formulas
\begin{equation}
u = (S_F)^{-1} S_K v,    \quad  v = \lambda^{-1} S u.   \lb{1.1c}
\end{equation}

Next, we will go a step further and describe a unitary equivalence result going beyond the connection between the eigenvalue problems \eqref{1.1a} and \eqref{1.1b}: Given $S$, we introduce the following sesquilinear forms in $\cH$,
\begin{align}
a(u,v) & = (Su,Sv)_{\cH}, \quad u, v \in \dom(a) = \dom(S),    \lb{1.2} \\
b(u,v) & = (u,Sv)_{\cH}, \quad u, v \in  \dom(b) = \dom(S).    \lb{1.3} 
\end{align}
Then $S$ being densely defined and closed, implies that the sesquilinear form $a$ is also densely defined and closed, and thus one can introduce the Hilbert space 
\begin{equation}
\cW=(\dom(S), (\cdot,\cdot)_{\cW})    \lb{1.4}
\end{equation}
with associated scalar product 
\begin{equation}
(u,v)_{\cW}=a(u,v) = (Su,Sv)_{\cH}, \quad u, v \in \dom(S).   \lb{1.5}
\end{equation}
Suppressing for simplicity the continuous embedding operator of $\cW$ 
into $\cH$, we now introduce the following operator $T$ in $\cW$ by
\begin{align} 
(w_1,T w_2)_{\cW} & = a( w_1,T w_2) = b(w_1,w_2) 
= (w_1,S w_2)_{\cH},  \quad w_1, w_2 \in \cW.    \lb{1.8}
\end{align}
One can prove that $T$ is self-adjoint, nonnegative, and bounded and we will call $T$ the 
{\it abstract buckling problem operator} associated with the Krein--von Neumann extension $S_K$ of $S$.  

Next, introducing the Hilbert space $\hatt \cH$ by 
\begin{equation} 
\hatt \cH = [\ker (S^*)]^{\bot} = \big[I_{\cH} - P_{\ker(S^*)}\big] \cH  
= \big[I_{\cH} - P_{\ker(S_K)}\big] \cH = [\ker (S_K)]^{\bot}, 
\end{equation}
where $P_{\cM}$ denotes the orthogonal projection onto the subspace 
$\cM \subset \cH$, we introduce the operator
\begin{equation}
\hatt S: \begin{cases} \cW \to \hatt \cH,  \\
w \mapsto S w,  \end{cases}      \lb{1.12}
\end{equation}
and note that $\hatt S\in\cB(\cW,\hatt \cH)$ maps $\cW$ unitarily onto $\hatt \cH$. 

Finally, defining the {\it reduced Krein--von Neumann operator}  
$\hatt S_K$ in $\hatt \cH$ by 
\begin{equation} 
\hatt S_K:=S_K|_{[\ker(S_K)]^{\bot}}   \, \text{ in $\hatt \cH$,}   \label{1.13}
\end{equation} 
we can state the principal unitary equivalence result to be proved in Theorem \ref{t3.3}:

The inverse of the reduced Krein--von Neumann operator $\hatt S_K$ in 
$\hatt \cH$ and the abstract buckling problem operator $T$ in $\cW$ are unitarily equivalent, 
\begin{equation}
\big(\hatt S_K\big)^{-1} = \hatt S T (\hatt S)^{-1}.    \lb{1.20}
\end{equation}
In addition, 
\begin{equation}
\big(\hatt S_K\big)^{-1} = U_S \big[|S|^{-1} S |S|^{-1}\big] (U_S)^{-1}.    \lb{1.20a}
\end{equation}

Here we used the polar decomposition of $S$, 
\begin{equation}
S = U_S |S|,  \, \text{ with } \,  
|S| = (S^* S)^{1/2} \geq \varepsilon I_{\cH}, \; \varepsilon > 0,  \, \text{ and } \, 
U_S \in \cB\big(\cH,\hatt \cH\big) \, \text{ unitary,}    \lb{1.19b}
\end{equation}
and one observes that the operator $|S|^{-1} S |S|^{-1}\in\cB(\cH)$ in \eqref{1.20a} 
is self-adjoint in $\cH$. 

As discussed at the end of Section \ref{s3}, one can readily rewrite the abstract linear pencil buckling eigenvalue problem \eqref{1.1b}, $S^* S u = \lambda S u$, 
$\lambda \neq 0$, in the form of the standard eigenvalue problem
$|S|^{-1} S |S|^{-1} w = \lambda^{-1} w$, $\lambda \neq 0$, $w = |S| u$, and hence establish the connection between \eqref{1.1a}, \eqref{1.1b} and \eqref{1.20}, 
\eqref{1.20a}.  
 
As mentioned in the abstract, the concrete case where $S$ is given by  
$S=\overline{-\Delta|_{C_0^\infty(\Omega)}}$ in $L^2(\Omega; d^n x)$, then yields the spectral equivalence between the inverse of the reduced Krein--von Neumann extension $\hatt S_K$ of $S$ and the problem of the buckling of a clamped plate. 
More generally, Grubb \cite{Gr83} actually treated the case where $S$ is generated by an appropriate elliptic differential expression of order $2m$, $m\in\bbN$, and also introduced the higher-order analog of the buckling problem; we briefly summarize this in Example \ref{e3.6}.

\section{The Abstract Krein--von Neumann Extension}
\label{s2}

To get started, we briefly elaborate on the notational conventions used
throughout this paper and especially throughout this section which collects abstract material on the Krein--von Neumann extension. Let $\cH$ be a separable complex Hilbert space, $(\cdot,\cdot)_{\cH}$ the scalar product in $\cH$ (linear in
the second factor), and $I_{\cH}$ the identity operator in $\cH$.
Next, let $T$ be a linear operator mapping (a subspace of) a
Banach space into another, with $\dom(T)$, $\ran(T)$, and $\ker(T)$ denoting the
domain, range, and kernel (i.e., null space) of $T$. The closure of a closable 
operator $S$ is denoted by $\ol S$. The spectrum, essential spectrum, discrete spectrum, and resolvent set of a closed linear operator in $\cH$ will be denoted by 
$\sigma(\cdot)$, $\sigma_{\rm ess}(\cdot)$, $\sigma_{\rm d}(\cdot)$, and 
$\rho(\cdot)$, respectively. The
Banach spaces of bounded and compact linear operators in $\cH$ are
denoted by $\cB(\cH)$ and $\cB_\infty(\cH)$, respectively. Similarly,
the Schatten--von Neumann (trace) ideals will subsequently be denoted
by $\cB_p(\cH)$, $p\in (0,\infty)$. Analogous notation $\cB(\cH_1,\cH_2)$,
$\cB_\infty (\cH_1,\cH_2)$, etc., will be used for bounded, compact,
etc., operators between two Hilbert spaces $\cH_1$ and $\cH_2$. 
Whenever applicable,
we retain the same type of notation in the context of Banach spaces.
Moreover, $\cX_1\hookrightarrow \cX_2$ denotes the continuous embedding
of the Banach space $\cX_1$ into the Banach space $\cX_2$. $\cX_1 \dotplus \cX_2$ 
denotes the (not necessarily orthogonal) direct sum of the subspaces $\cX_1$ and 
$\cX_2$ of $\cX$. 

A linear operator $S:\dom(S)\subseteq\cH\to\cH$, is called {\it symmetric}, if
\begin{equation}\label{Pos-2}
(u,Sv)_\cH=(Su,v)_\cH, \quad u,v\in \dom (S).
\end{equation}
In this manuscript we will be particularly interested in this question within the class of 
densely defined (i.e., $\ol{\dom(S)}=\cH$), non-negative operators (in fact, in most instances $S$ will even turn out to be strictly positive) and we focus almost exclusively on self-adjoint extensions that are non-negative operators. In the latter scenario, there are two distinguished constructions which we review briefly next. 

To set the stage, we recall that a linear operator $S:\dom(S)\subseteq\cH\to \cH$
is called {\it non-negative} provided
\begin{equation}\label{Pos-1}
(u,Su)_\cH\geq 0, \quad u\in \dom(S).
\end{equation}
(In particular, $S$ is symmetric in this case.) $S$ is called {\it strictly positive}, if for some 
$\varepsilon >0$, $(u,Su)_\cH\geq \varepsilon \|u\|_{\cH}^2$, $u\in \dom(S)$. Next, we recall that $A \leq B$ for two self-adjoint operators in $\cH$ if 
\begin{align}
\begin{split}
& \dom\big(|A|^{1/2}\big) \supseteq \dom\big(|B|^{1/2}\big) \, \text{ and } \\ 
& \big(|A|^{1/2}u,U_A |A|^{1/2}u\big)_{\cH} \leq \big(|B|^{1/2}u, U_B |B|^{1/2}u\big)_{\cH}, \quad  
u \in \dom\big(|B|^{1/2}\big).      \lb{AleqB} 
\end{split}
\end{align}
Here $U_C$ denotes the partial isometry in $\cH$ in the polar decomposition 
$C = U_C |C|$, $|C|=(C^* C)^{1/2}$, of a densely defined closed operator $C$ in $\cH$. (If $C$ is in addition self-adjoint, then $|C|$ and $U_C$ commute.) 
We also recall that for $A\geq 0$ self-adjoint, 
\begin{equation}
\ker(A) =\ker\big(A^{1/2}\big)
\end{equation}
(with $D^{1/2}$ denoting the unique nonnegative square root of a nonnegative self-adjoint operator $D$ in $\cH$).

For simplicity we will always adhere to the conventions that $S$ is a linear, unbounded, densely defined, nonnegative (i.e., $S\geq 0$) operator in $\cH$, and that $S$ has nonzero deficiency indices. Since $S$ is bounded from below, the latter are necessarily equal. In particular,  
\begin{equation}
{\rm def} (S) = \dim (\ker(S^*-z I_{\cH})) \in \bbN\cup\{\infty\}, 
\quad z\in \bbC\backslash [0,\infty), 
\lb{DEF}
\end{equation}
is well-known to be independent of $z$.  
Moreover, since $S$ and its closure $\ol{S}$ have the same self-adjoint extensions in $\cH$, we will without loss of generality assume that $S$ is closed in the remainder of this paper.

The following is a fundamental result to be found in M.\ Krein's celebrated 1947 paper
 \cite{Kr47} (cf.\ also Theorems\ 2 and 5--7 in the English summary on page 492)\footnote{We are particularly indebted to Gerd Grubb for a clarification of the necessary and sufficient nature of 
the inequalities \eqref{Fr-Sa} (resp.,\ \eqref{Res}) for $\widetilde{S}$ to be a self-adjoint extension 
of $S$.}: 

\begin{theorem}\label{T-kkrr}
Assume that $S$ is a densely defined, closed, nonnegative operator in $\cH$. Then, among all non-negative self-adjoint extensions of $S$, there exist two distinguished ones, $S_K$ and $S_F$, which are, respectively, the smallest and largest
$($in the sense of order between self-adjoint operators, cf.\ \eqref{AleqB}$)$ such extension. Furthermore, a non-negative self-adjoint operator $\widetilde{S}$ is a self-adjoint extension 
of $S$ if and only if $\widetilde{S}$ satisfies 
\begin{equation}\label{Fr-Sa}
S_K\leq\widetilde{S}\leq S_F.
\end{equation}
In particular, \eqref{Fr-Sa} determines $S_K$ and $S_F$ uniquely. 

In addition,  if $S\geq \varepsilon I_{\cH}$ for some $\varepsilon >0$, one has 
$S_F \geq \varepsilon I_{\cH}$, and 
\begin{align}
\dom (S_F) &= \dom (S) \dotplus (S_F)^{-1} \ker (S^*),     \lb{SF}  \\
\dom (S_K) & = \dom (S) \dotplus \ker (S^*),    \lb{SK}   \\
\dom (S^*) & = \dom (S) \dotplus (S_F)^{-1} \ker (S^*) \dotplus \ker (S^*)  \no \\
& = \dom (S_F) \dotplus \ker (S^*),    \lb{S*} 
\end{align}
in particular, 
\begin{equation} \label{Fr-4Tf}
\ker(S_K)= \ker\big((S_K)^{1/2}\big)= \ker(S^*) = \ran(S)^{\bot}.
\end{equation} 
\end{theorem}

We also note that
\begin{align}
S_F u & = S^* u, \quad u \in \dom(S_F),   \\
S_K v & = S^* v, \quad v \in \dom(S_K). 
\end{align}

Here the operator inequalities in \eqref{Fr-Sa} are understood in the sense of \eqref{AleqB} and they can  equivalently be written as
\begin{equation}
(S_F + a I_{\cH})^{-1} \le \big(\wti S + a I_{\cH}\big)^{-1} \le (S_K + a I_{\cH})^{-1} 
\, \text{ for some (and hence for all) $a > 0$.}    \lb{Res}
\end{equation}

For classical references on the subject of self-adjoint extensions of semibounded operators (not necessarily restricted to the Krein--von Neumann extension) we refer to Birman \cite{Bi56}, \cite{Bi08}, Friedrichs \cite{Fr34}, Freudenthal \cite{Fr36}, Grubb \cite{Gr68}, \cite{Gr70}, Krein \cite{Kr47a}, 
{\u S}traus \cite{St73}, and Vi{\u s}ik \cite{Vi63} (see also the monographs by Akhiezer and Glazman 
\cite[Sect. 109]{AG81a}, Faris \cite[Part III]{Fa75}, Fukushima, Oshima, and Takeda 
\cite[Sect.\ 3.3]{FOT94}, and the recent book by Grubb \cite[Sect.\ 13.2]{Gr09}).  

We will call the operator $S_K$ the {\it Krein--von Neumann extension}
of $S$. See \cite{Kr47} and also the discussion in \cite{AS80} and \cite{AN70}. It should be noted that the Krein--von Neumann extension was first considered by von Neumann 
\cite{Ne29} in 1929 in the case where $S$ is strictly bounded from below, that is, if 
$S \geq \varepsilon I_{\cH}$ for some $\varepsilon >0$. (His construction appears in the proof of Theorem 42 on pages 102--103.) However, von Neumann did not isolate the extremal property of this extension as described in \eqref{Fr-Sa} and \eqref{Res}. 
M.\ Krein \cite{Kr47}, \cite{Kr47a} was the first to systematically treat the general case $S\geq 0$ and to study all nonnegative self-adjoint extensions of $S$, illustrating the special role of the {\it Friedrichs extension} (i.e., the ``hard'' extension) $S_F$ of $S$ and the Krein--von Neumann (i.e., the ``soft'') extension $S_K$ of $S$ as extremal cases when considering all nonnegative extensions of $S$. For a recent exhaustive treatment of self-adjoint extensions of semibounded operators we refer to 
\cite{AT02}--\cite{AT09}, \cite{DM91}, \cite{DM95}, \cite{HMD04}. 

For convenience of the reader we also mention the following intrinsic description of the Friedrichs extension $S_F$ of $S\geq 0$ ($S$ densely defined and closed in $\cH$) due to Freudenthal \cite{Fr36}, 
\begin{align}
& S_F u:=S^*u,   \no \\
& u \in \dom(S_F):=\big\{v\in\dom(S^*)\,\big|\,  \mbox{there exists} \, 
\{v_j\}_{j\in\bbN}\subset \dom(S),    \label{Fr-2} \\
& \quad \mbox{with} \, \lim_{j\to\infty}\|v_j-v\|_{\cH}=0  
\mbox{ and } ((v_j-v_k),S(v_j-v_k))_\cH\to 0 \mbox{ as }  j,k\to\infty\big\},   \no 
\end{align}
and an intrinsic description of the Krein--von Neumann extension $S_K$ of $S\geq 0$ due to Ando and Nishio \cite{AN70}, 
\begin{align} 
& S_Ku:=S^*u,   \no \\
& u \in \dom(S_K):=\big\{v\in\dom(S^*)\,\big|\,\mbox{there exists} \, 
\{v_j\}_{j\in\bbN}\subset \dom(S),    \label{Fr-2X}  \\ 
& \quad \mbox{with} \, \lim_{j\to\infty} \|Sv_j-S^*v\|_{\cH}=0  
\mbox{ and } ((v_j-v_k),S(v_j-v_k))_\cH\to 0 \mbox{ as } j,k\to\infty\big\}.  \no
\end{align}

Throughout the rest of this paper we make the following assumptions: 

\begin{hypothesis}  \lb{h2.2}
Suppose that $S$ is a densely defined, symmetric, closed operator 
with nonzero deficiency indices in $\cH$ that satisfies 
\begin{equation}
S\geq \varepsilon I_{\cH} \, \text{ for some $\varepsilon >0$.}    \lb{3.1}
\end{equation}
\end{hypothesis}

We recall that the {\it reduced Krein--von Neumann operator} 
$\hatt S_K$ in the Hilbert space $\hatt \cH$ (cf.\ \eqref{Fr-4Tf}), 
\begin{equation}
\hatt \cH = [\ker (S^*)]^{\bot} = \big[I_{\cH} - P_{\ker(S^*)}\big] \cH 
= \big[I_{\cH} - P_{\ker(S_K)}\big] \cH = [\ker (S_K)]^{\bot},    \lb{hattH}
\end{equation}
is given by 
\begin{align} \label{Barr-4}
\hatt S_K:&=S_K|_{[\ker(S_K)]^{\bot}} \\
\begin{split} 
& = S_K[I_{\cH} - P_{\ker(S_K)}]  \, \text{ in $\hatt \cH$}    \lb{SKP} \\
&= [I_{\cH} - P_{\ker(S_K)}]S_K[I_{\cH} - P_{\ker(S_K)}] 
\, \text{ in $\hatt \cH$},  
\end{split}
\end{align} 
where $P_{\cM}$ denotes the orthogonal projection onto the subspace 
$\cM \subset \cH$, and we are alluding to the orthogonal direct sum decomposition of 
$\cH$ into 
\begin{equation}
\cH = P_{\ker(S_K)}\cH \oplus \hatt \cH = \ker(S_K) \oplus [\ker(S_K)]^\bot.
\end{equation}

We continue with the following elementary observation: 

\begin{lemma}  \lb{l3.1a}
Assume Hypothesis \ref{h2.2} and let $v\in\dom(S_K)$. Then the decomposition, 
$\dom(S_K)= \dom(S) \dotplus \ker(S^*)$ $($cf.\ \eqref{SK}$)$, leads to the following decomposition of $v$, 
\begin{equation}
v= (S_F)^{-1} S_K v + w, \, \text{ where $(S_F)^{-1} S_K v \in \dom(S)$ and 
$w \in \ker(S^*)$.}     \lb{3.1aa}
\end{equation}
As a consequence,
\begin{equation}
\big(\hatt S_K\big)^{-1} = [I_{\cH} - P_{\ker(S_K)}] (S_F)^{-1} [I_{\cH} - P_{\ker(S_K)}].   
\lb{SKinv}
\end{equation}
\end{lemma}
\begin{proof}
Let $v = u + w$, with $u \in \dom(S)$ and $w \in \ker(S^*)$. Then
\begin{align}
v & = u + w = (S_F)^{-1} S_F u + w = (S_F)^{-1} S u + w   \no \\
& = (S_F)^{-1} S_K u +w = (S_F)^{-1} S_K (u + w) +w   \no \\
& = (S_F)^{-1} S_K v +w     \lb{3.1a}
\end{align}
proves \eqref{3.1aa}. Given $v\in\dom(S_K)$, one infers 
\begin{equation}
S_K v = S_K (P_{\ker(S_K)} + P_{\hatt\cH})v = S_K P_{\hatt\cH} v,
\end{equation}
since $S_K P_{\ker(S_K)}=0$. In particular,
\begin{equation}
P_{\hatt \cH} v \in \dom (S_K) \, \text{ whenever } \, v\in\dom(S_K).
\end{equation} 
Applying $P_{\hatt \cH}$ to \eqref{3.1aa} then yields 
\begin{align}
P_{\hatt \cH} v & = P_{\hatt \cH}  (S_F)^{-1} S_K [P_{\hatt \cH} + P_{\ker(S_K)}] v 
= P_{\hatt \cH} (S_F)^{-1} S_K P_{\hatt \cH} v 
= P_{\hatt \cH} (S_F)^{-1} \hatt S_K P_{\hatt \cH} v \no \\
& = P_{\hatt \cH} (S_F)^{-1} P_{\hatt \cH} \hatt S_K P_{\hatt \cH} v, 
\quad v\in\dom(S_K).  
\end{align}
Thus,
\begin{equation}
\big(\hatt S_K\big)^{-1} \big( \hatt S_K P_{\hatt \cH} v\big) = P_{\hatt \cH} (S_F)^{-1} P_{\hatt \cH} \big(\hatt S_K P_{\hatt \cH} v\big), \quad v\in\dom(S_K).  \lb{3.1bb}
\end{equation}
Since $\ran\big(\hatt S_K\big) = \hatt \cH$, \eqref{3.1bb} proves \eqref{SKinv}. 
\end{proof}

We note that equation \eqref{SKinv} was proved by Krein in his seminal paper 
\cite{Kr47} (cf.\ the proof of Theorem\ 26 in \cite{Kr47}). For a different proof of Krein's formula 
\eqref{SKinv} and its generalization to the case of non-negative operators, see also 
\cite[Corollary 5]{Ma92}.

Next, we consider a self-adjoint operator
\begin{equation} \label{Barr-1}
T:\dom(T)\subseteq \cH\to\cH,\quad T=T^*,
\end{equation} 
which is bounded from below, that is, there exists $\alpha\in\bbR$ such that
\begin{equation} \label{Barr-2}
T\geq \alpha I_{\cH}.
\end{equation} 
We denote by $\{E_T(\lambda)\}_{\lambda\in\bbR}$ the family of strongly right-continuous spectral projections of $T$, and introduce, as usual, $E_T((a,b))=E_T(b_-) - E_T(a)$, 
$E_T(b_-) = \slim_{\varepsilon\downarrow 0}E_T(b-\varepsilon)$, $-\infty \leq a < b$. In addition, we set 
\begin{equation} \label{Barr-3}
\mu_{T,j}:=\inf\,\bigl\{\lambda\in\bbR\,|\,
\dim (\ran (E_T((-\infty,\lambda)))) \geq j\bigr\},\quad j\in\bbN.
\end{equation} 
Then, for fixed $k\in\bbN$, either: \\
$(i)$ $\mu_{T,k}$ is the $k$th eigenvalue of $T$ counting multiplicity below the bottom of the essential spectrum, $\sigma_{\rm ess}(T)$, of $T$, \\
or \\
$(ii)$ $\mu_{T,k}$ is the bottom of the essential spectrum of $T$, 
\begin{equation}
\mu_{T,k} = \inf \{\lambda \in \bbR \,|\, \lambda \in \sigma_{\rm ess}(T)\}, 
\end{equation}
and in that case $\mu_{T,k+\ell} = \mu_{T,k}$, $\ell\in\bbN$, and there are at most $k-1$ eigenvalues (counting multiplicity) of $T$ below $\mu_{T,k}$. 

We now record the following basic result: 

\begin{theorem}  \lb{tKAS}
Assume Hypothesis \ref{h2.2}. Then, 
\begin{equation}\label{Barr-5}
\varepsilon \leq \mu_{S_F,j} \leq \mu_{\hatt S_K,j}, \quad j\in\bbN.
\end{equation} 
In particular, if the Friedrichs extension $S_F$ of $S$ has purely discrete
spectrum, then, except possibly for $\lambda=0$, the Krein--von Neumann extension
$S_K$ of $S$ also has purely discrete spectrum in $(0,\infty)$, that is, 
\begin{equation}
\sigma_{\rm ess}(S_F) = \emptyset \, \text{ implies } \, 
\sigma_{\rm ess}(S_K) \backslash\{0\} = \emptyset.      \lb{ESSK}
\end{equation}
In addition, let $p\in (0,\infty)\cup\{\infty\}$, then
\begin{align}
\begin{split}
& (S_F - z_0 I_{\cH})^{-1} \in \cB_p(\cH) 
\, \text{ for some $z_0\in \bbC\backslash [\varepsilon,\infty)$}   \\ 
& \text{implies } \, 
(S_K - zI_{\cH})^{-1}[I_{\cH} - P_{\ker(S_K)}] \in \cB_p(\cH) 
 \, \text{ for all $z\in \bbC\backslash [\varepsilon,\infty)$}. 
\lb{CPK}
\end{split} 
\end{align}
In fact, the $\ell^p(\bbN)$-based trace ideals $\cB_p(\cH)$ of $\cB(\cH)$ can be replaced by any two-sided symmetrically normed ideals of $\cB(\cH)$.
\end{theorem}
\begin{proof}
Denote by $\cM_j$ subspaces of $\cH$ of dimension $j\in\bbN$, and similarly, $\hatt \cM_j$ subspaces of $\hatt \cH$ of dimension $j\in\bbN$. Then the  inequalities 
\eqref{Barr-5} follow from $S_F \geq \varepsilon I_{\cH}$, \eqref{SKinv}, and the minimax (better, maximin) theorem as follows: First we note that (cf., e.g., 
\cite[Theorem\ 5.28]{GS06}, \cite[Sect.\ 32]{He86})
\begin{equation}
\f{1}{\mu_{S_F,j}} = \sup_{\cM_j \subset \cH}
\min_{\substack{u \in \cM_j\\ \|u\|_{\cH}=1}} \big(u, (S_F)^{-1} u\big)_{\cH}, \quad 
j\in\bbN.
\end{equation}
As a consequence, 
\begin{equation}
\f{1}{\mu_{S_F,j}} \geq 
\min_{u \in \cM_j \subset \cH} \big(u, (S_F)^{-1} u\big)_{\cH}, \quad j\in\bbN, 
\end{equation}
for any subspace $\cM_j$ of $\cH$ of dimension $j\in\bbN$. In particular, 
\begin{align}
\f{1}{\mu_{S_F,j}} &\geq 
\min_{\substack{v \in \hatt \cM_j \subset \hatt\cH\\ \|v\|_{\hatt \cH}=1}} 
\big(v, (S_F)^{-1} v\big)_{\hatt \cH}  \no \\
& = \min_{\substack{v \in \hatt \cM_j \subset \hatt\cH\\ \|v\|_{\hatt \cH}=1}} 
\big(v, P_{\hatt \cH} (S_F)^{-1} P_{\hatt \cH}v\big)_{\hatt \cH},  \quad j\in\bbN, 
\end{align}
for any subspace $\hatt \cM_j$ of $\hatt \cH$ of dimension $j\in\bbN$. Thus, one concludes 
\begin{align}
\f{1}{\mu_{S_F,j}} 
& \geq  \sup_{\hatt \cM_j \subset \hatt \cH} \; 
\min_{\substack{v \in \hatt \cM_j\\ \|v\|_{\hatt \cH}=1}} 
\big(v, P_{\hatt \cH} (S_F)^{-1} P_{\hatt \cH}v\big)_{\hatt \cH}    \no \\ 
& = \sup_{\hatt \cM_j \subset \hatt \cH} \; 
\min_{\substack{v \in \hatt \cM_j\\ \|v\|_{\hatt \cH}=1}} 
\big(v, \big(\hatt S_K\big)^{-1} v\big)_{\hatt \cH}    \no \\ 
& = \f{1}{\mu_{\hatt S_K,j}}, \quad j\in\bbN. 
\end{align}
Next, let $\cJ(\cH)$ be a two-sided symmetrically normed ideal of $\cB(\cH)$. Temporarily, we will identify operators of the type $P_{\hatt\cH} TP_{\hatt\cH}$ in 
$\hatt \cH$ for $T \in \cB(\cH)$, with $2 \times 2$ block operators of the type
\begin{equation}
\begin{pmatrix} 0 & 0 \\ 0 & P_{\hatt\cH} TP_{\hatt\cH}|_{\hatt \cH} \end{pmatrix}  
\, \text{ in } \, \cH = (\ker(S_K))^{\bot} \oplus \hatt \cH.
\end{equation}  
By \eqref{SKinv}, and since $P_{\hatt\cH}$ is bounded, one concludes that  
$(S_F)^{-1} \in \cJ(\cH)$ implies 
$\big(\hatt S_K)^{-1} = \nlim_{z\to 0}(S_K - zI_{\cH})^{-1}[I_{\cH} - P_{\ker(S_K)}] \in \cJ(\cH)$. The (first) resolvent equation applied to $S_F$, and subsequently, applied to $S_K$, then proves \eqref{CPK}.
\end{proof}

We note that \eqref{ESSK} is a classical result of Krein \cite{Kr47}, the more general fact \eqref{Barr-5} has not been mentioned explicitly in Krein's paper \cite{Kr47}, although it immediately follows from the minimax principle and Krein's formula \eqref{SKinv}. On the other hand, in the special case 
${\rm def}(S)<\infty$, Krein states an extension of \eqref{Barr-5} in his Remark 8.1 in the sense that he also considers self-adjoint extensions different from the Krein extension. Apparently, \eqref{Barr-5} has first been proven by Alonso and Simon \cite{AS80} by a somewhat different method. 

Concluding this section, we point out that a great variety of additional results 
for the Krein--von Neumann extension can be found in the very extensive list of 
references in \cite{AT09}, \cite{AGMT09}, and \cite{HMD04}.

\section{The Krein--von Neumann Extension and its Unitary Equivalence to an Abstract Buckling Problem}
\label{s3}

In this section we prove our principal result, the unitary equivalence of the inverse of 
the Krein--von Neumann extension (on the orthogonal complement of its kernel) of a densely defined, closed, operator $S$ satisfying $S\geq \varepsilon I_{\cH}$ for some 
$\varepsilon >0$, in a complex separable Hilbert space $\cH$ to an abstract buckling problem operator.  

We start by introducing an abstract version of Proposition\ 1 in Grubb's paper  
\cite{Gr83} devoted to Krein--von Neumann extensions of even order elliptic differential operators on bounded domains:

\begin{lemma}  \lb{l3.3}
Assume Hypothesis \ref{h2.2} and let $\lambda \neq 0$. Then there exists 
$0 \neq v \in \dom(S_K)$ with
\begin{equation}
S_K v = \lambda v   \lb{3.1b}
\end{equation}
if and only if there exists $0 \neq u \in \dom(S^* S)$ such that
\begin{equation}
S^* S u = \lambda S u.   \lb{3.1c}
\end{equation}
In particular, the solutions $v$ of \eqref{3.1b} are in one-to-one correspondence with the solutions $u$ of \eqref{3.1c} given by the formulas
\begin{align}
u & = (S_F)^{-1} S_K v,    \lb{3.1d}  \\
v & = \lambda^{-1} S u.   \lb{3.1e}
\end{align}
Of course, since $S_K \geq 0$, any $\lambda \neq 0$ in \eqref{3.1b} and \eqref{3.1c}  necessarily satisfies $\lambda > 0$.
\end{lemma}
\begin{proof}
Let $S_K v = \lambda v$, $v\in\dom(S_K)$, $\lambda \neq 0$, and $v = u + w$, with 
$u \in \dom(S)$ and $w \in \ker(S^*)$. Then, 
\begin{equation}
S_K v = \lambda v 
\Longleftrightarrow v = \lambda^{-1} S_K v = \lambda^{-1} S_K u = \lambda^{-1} S u.
\end{equation}
Moreover, $u =0$ implies $v = 0$ and clearly $v=0$ implies $u=w=0$, hence $v \neq 0$ if and only if $u \neq 0$. In addition, $u = (S_F)^{-1} S_K v$ by \eqref{3.1aa}. Finally,
\begin{align}
\begin{split}
& \lambda w = S u - \lambda u \in \ker(S^*) \, \text{ implies}    \\ 
& 0 = \lambda S^* w = S^*(S u - \lambda u) = S^* S u - \lambda S^* u 
= S^* S u - \lambda S u. 
\end{split}
\end{align}
Conversely, suppose $u \in \dom(S^* S)$ and $S^* S u = \lambda S u$, 
$\lambda \neq 0$. Introducing $v = \lambda^{-1} S u$, then $v \in \dom(S^*)$ and 
\begin{equation}
S^* v = \lambda^{-1} S^* S u = S u = \lambda v.
\end{equation}
Noticing that
\begin{equation}
S^* S u = \lambda S u = \lambda S^* u \, \text{ implies } \, S^*(S-\lambda I_{\cH}) u =0, 
\end{equation}
and hence $(S - \lambda I_{\cH}) u \in \ker(S^*)$, rewriting $v$ as
\begin{equation}
v = u + \lambda^{-1} (S - \lambda I_{\cH}) u
\end{equation}
then proves that also $v \in \dom(S_K)$, using \eqref{SK} again.
\end{proof}

Due to Example \ref{e3.6} and Remark \ref{r3.7} at the end of this section, we will call the linear pencil eigenvalue problem $S^* Su = \lambda S u$ in \eqref{3.1c} the {\it abstract buckling problem} associated with the Krein--von Neumann extension $S_K$ of $S$.

Next, we turn to a variational formulation of the correspondence between the inverse of the reduced Krein extension $\hatt S_K$ and the abstract buckling problem in terms of appropriate sesquilinear forms by following the treatment of Kozlov 
\cite{Ko79}--\cite{Ko84} in the context of elliptic partial differential operators. This will then lead to an even stronger connection between the Krein--von Neumann extension $S_K$ of $S$ and the associated abstract buckling eigenvalue problem \eqref{3.1c}, culminating in a unitary equivalence result in Theorem \ref{t3.3}. 

Given the operator $S$, we introduce the following sesquilinear forms in $\cH$,
\begin{align}
a(u,v) & = (Su,Sv)_{\cH}, \quad u, v \in \dom(a) = \dom(S),    \lb{3.2} \\
b(u,v) & = (u,Sv)_{\cH}, \quad u, v \in  \dom(b) = \dom(S).    \lb{3.3} 
\end{align}
Then $S$ being densely defined and closed implies that the sesquilinear form $a$ shares these properties and  \eqref{3.1} implies its boundedness from below,
\begin{equation}
a(u,u) \geq \varepsilon^2 \|u\|_{\cH}^2, \quad u \in \dom(S).    \lb{3.4}
\end{equation} 
Thus, one can introduce the Hilbert space 
$\cW=(\dom(S), (\cdot,\cdot)_{\cW})$ with associated scalar product 
\begin{equation}
(u,v)_{\cW}=a(u,v) = (Su,Sv)_{\cH}, \quad u, v \in \dom(S).   \lb{3.5}
\end{equation}
In addition, we denote by $\iota_{\cW}$ the continuous embedding operator of $\cW$ 
into $\cH$,
\begin{equation}
\iota_{\cW} : \cW \hookrightarrow \cH.    \lb{3.6}
\end{equation}
Hence we will use the notation
\begin{equation}
(w_1,w_2)_{\cW} =a(\iota_{\cW} w_1,\iota_{\cW} w_2) 
= (S\iota_{\cW} w_1, S\iota_{\cW} w_2)_{\cH}, \quad w_1, w_2 \in \cW,   \lb{3.7}
\end{equation}
in the following. 

Given the sesquilinear forms $a$ and $b$ and the Hilbert space $\cW$, we next define the operator $T$ in $\cW$ by
\begin{align}
\begin{split}
(w_1,T w_2)_{\cW} & = a(\iota_{\cW} w_1,\iota_{\cW} T w_2) 
= (S \iota_{\cW} w_1,S\iota_{\cW} T w_2)_{\cH}   \\
& = b(\iota_{\cW} w_1,\iota_{\cW} w_2) = (\iota_{\cW} w_1,S \iota_{\cW} w_2)_{\cH},
\quad w_1, w_2 \in \cW.    \lb{3.8}
\end{split}
\end{align}
(In contrast to the informality of our introduction, we now explicitly write the embedding operator $\iota_{\cW}$.) One verifies that $T$ is well-defined and that 
\begin{equation}
|(w_1,T w_2)_{\cW}| \leq \|\iota_{\cW} w_1\|_{\cH} \|S \iota_{\cW} w_2\|_{\cH} 
\leq \varepsilon^{-1} \|w_1\|_{\cW} \|w_2\|_{\cW}, \quad w_1, w_2 \in \cW,   \lb{3.9}
\end{equation}
and hence that
\begin{equation}
0 \leq T = T^* \in \cB(\cW), \quad \|T\|_{\cB(\cW)} \leq \varepsilon^{-1}.  \lb{3.10}
\end{equation}
For reasons to become clear at the end of this section, we will call $T$ the {\it abstract buckling problem operator} associated with the Krein--von Neumann extension $S_K$ of $S$.  

Next, recalling the notation 
$\hatt \cH = [\ker (S^*)]^{\bot} = \big[I_{\cH} - P_{\ker(S^*)}\big] \cH$ (cf.\ \eqref{hattH}), 
we introduce the operator
\begin{equation}
\hatt S: \begin{cases} \cW \to \hatt \cH,  \\
w \mapsto S \iota_{\cW} w,  \end{cases}      \lb{3.12}
\end{equation}
and note that 
\begin{equation} 
\ran\big(\hatt S\big) = \ran (S) = \hatt \cH,    \lb{3.12aa}
\end{equation}
since $S\geq \varepsilon I_{\cH}$ for some $\varepsilon > 0$ and $S$ is closed in $\cH$ 
(see, e.g., \cite[Theorem\ 5.32]{We80}). In fact, one has the following result:

\begin{lemma} \lb{l3.1} 
Assume Hypothesis \ref{h2.2}. Then 
$\hatt S\in\cB(\cW,\hatt \cH)$ maps $\cW$ unitarily onto $\hatt \cH$. 
\end{lemma}
\begin{proof}
Clearly $\hatt S$ is an isometry since 
\begin{equation}
\big\|\hatt S w\big\|_{\hatt\cH} = \|S \iota_{\cW} w\big\|_{\cH} = \|w\|_{\cW}, \quad w \in \cW.   \lb{3.13}
\end{equation}
Since $\ran\big(\hatt S\big) = \hatt \cH$ by \eqref{3.12aa}, $\hatt S$ is unitary. 
\end{proof}

Next we recall the definition of the reduced Krein--von Neumann operator $\hatt S_K$ in $\hatt \cH$ defined in \eqref{SKP}, the fact that $\ker(S^*) = \ker(S_K)$ by \eqref{Fr-4Tf}, and state the following auxiliary result:

\begin{lemma} \lb{l3.2}
Assume Hypothesis \ref{h2.2}. Then the map
\begin{equation}
\big[I_{\cH} - P_{\ker(S^*)}\big] : \dom(S) \to \dom \big(\hatt S_K\big)    \lb{3.15}
\end{equation}
is a bijection. In addition, we note that 
\begin{align}
\begin{split}
& \big[I_{\cH} - P_{\ker(S^*)}\big] S_K u = S_K \big[I_{\cH} - P_{\ker(S^*)}\big] u 
= \hatt S_K \big[I_{\cH} - P_{\ker(S^*)}\big] u  \\
& \quad = \big[I_{\cH} - P_{\ker(S^*)}\big] S u = Su \in \hatt \cH,    
\quad  u \in \dom(S).     \lb{3.16} 
\end{split}
\end{align}
\end{lemma}
\begin{proof}
Let $u\in\dom(S)$, then $\ker(S^*) = \ker(S_K)$ implies that 
$\big[I_{\cH} - P_{\ker(S^*)}\big] u \in \dom(S_K)$ and of course 
$\big[I_{\cH} - P_{\ker(S^*)}\big] u \in \dom\big(\hatt S_K\big)$. To prove injectivity of the map 
\eqref{3.15} it suffices to assume $v \in \dom(S)$ and 
$\big[I_{\cH} - P_{\ker(S^*)}\big] v =0$. Then 
$\dom(S) \ni v = P_{\ker(S^*)} v \in \ker(S^*)$ yields $v=0$ as 
$\dom(S) \cap \ker(S^*) = \{0\}$. To prove surjectivity of the map \eqref{3.15} we suppose $u \in \dom\big(\hatt S_K)$. The decomposition, $u = f +g$ with $f \in \dom(S)$ and $g \in \ker(S^*)$, then yields 
\begin{equation}
u = \big[I_{\cH} - P_{\ker(S^*)}\big] u = \big[I_{\cH} - P_{\ker(S^*)}\big] f 
\in \big[I_{\cH} - P_{\ker(S^*)}\big] \dom(S)   \lb{3.18}
\end{equation}
and hence proves surjectivity of \eqref{3.15}. 

Equation \eqref{3.16} is clear from 
\begin{equation}
S_K \big[I_{\cH} - P_{\ker(S^*)}\big] = \big[I_{\cH} - P_{\ker(S^*)}\big] S_K 
= \big[I_{\cH} - P_{\ker(S^*)}\big] S_K \big[I_{\cH} - P_{\ker(S^*)}\big].   \lb{3.19}
\end{equation} 
\end{proof}

Continuing, we briefly recall the polar decomposition of $S$, 
\begin{equation}
S = U_S |S|,   \lb{3.19a}
\end{equation} 
with
\begin{equation}
|S| = (S^* S)^{1/2} \geq \varepsilon I_{\cH}, \; \varepsilon > 0, \quad 
U_S \in \cB\big(\cH,\hatt \cH\big) \, \text{ is unitary.}    \lb{3.19b}
\end{equation}

At this point we are in position to state our principal unitary equivalence result:

\begin{theorem} \lb{t3.3}
Assume Hypothesis \ref{h2.2}. Then the inverse of the reduced Krein--von Neumann extension $\hatt S_K$ in $\hatt \cH = \big[I_{\cH} - P_{\ker(S^*)}\big] \cH$ and the abstract buckling problem operator $T$ in $\cW$ are unitarily equivalent, in particular,
\begin{equation}
\big(\hatt S_K\big)^{-1} = \hatt S T (\hatt S)^{-1}.    \lb{3.20}
\end{equation}
Moreover, one has
\begin{equation}
\big(\hatt S_K\big)^{-1} = U_S \big[|S|^{-1} S |S|^{-1}\big] (U_S)^{-1},    \lb{3.20a}
\end{equation}
where $U_S\in \cB\big(\cH,\hatt \cH\big)$ is the unitary operator in the polar decomposition \eqref{3.19a} of $S$ and the operator $|S|^{-1} S |S|^{-1}\in\cB(\cH)$ is self-adjoint in $\cH$. 
\end{theorem}
\begin{proof}
Let $w_1, w_2 \in \cW$. Then,
\begin{align}
& \big(w_1,\big(\hatt S\big)^{-1} \big(\hatt S_K\big)^{-1} \hatt S w_2\big)_{\cW} 
= \big(\hatt S w_1, \big(\hatt S_K\big)^{-1} \hatt S w_2\big)_{\hatt \cH}  \no \\
& \quad = \big( \big(\hatt S_K\big)^{-1} \hatt S w_1, \hatt S w_2\big)_{\hatt \cH} 
= \big( \big(\hatt S_K\big)^{-1} S \iota_{\cW} w_1, \hatt S w_2\big)_{\hatt \cH}  \no \\
& \quad = \big( \big(\hatt S_K\big)^{-1} \big[I_{\cH} - P_{\ker(S^*)}\big] S \iota_{\cW} w_1, \hatt S w_2\big)_{\hatt \cH}   \quad 
 \text{by \eqref{3.16}}   \no \\
& \quad = \big( \big(\hatt S_K\big)^{-1} \hatt S_K \big[I_{\cH} - P_{\ker(S^*)}\big] 
\iota_{\cW} w_1, \hatt S w_2\big)_{\hatt \cH}  \quad 
 \text{again by \eqref{3.16}}  \no \\
&  \quad = \big(\big[I_{\cH} - P_{\ker(S^*)}\big] 
\iota_{\cW} w_1, \hatt S w_2\big)_{\hatt \cH}    \no \\
&  \quad = \big(\iota_{\cW} w_1, S \iota_{\cW} w_2\big)_{\cH}    \no \\
&  \quad = \big(w_1, T w_2\big)_{\cW}  \quad 
\text{by definition of $T$ in \eqref{3.8},}    \lb{3.22}
\end{align}
yields \eqref{3.20}. In addition one verifies that
\begin{align}
\big(\hatt S w_1, \big(\hatt S_K\big)^{-1} \hatt S w_2\big)_{\hatt \cH} & = 
\big(w_1, T w_2\big)_{\cW}   \no \\
& = \big(\iota_{\cW} w_1, S \iota_{\cW} w_2\big)_{\cH}    
\no \\ 
& = \big(|S|^{-1} |S| \iota_{\cW} w_1, S |S|^{-1} |S| \iota_{\cW} w_2\big)_{\cH}    \no \\
& = \big(|S| \iota_{\cW} w_1, \big[|S|^{-1} S |S|^{-1}\big] |S| \iota_{\cW} w_2\big)_{\cH}    \no \\
& = \big((U_S)^* S \iota_{\cW} w_1, \big[|S|^{-1} S |S|^{-1}\big] (U_S)^* S \iota_{\cW} w_2\big)_{\cH}    \no \\
& = \big(S \iota_{\cW} w_1, U_S \big[|S|^{-1} S |S|^{-1}\big] (U_S)^* S \iota_{\cW} w_2\big)_{\cH}    \no \\
& = \big(\hatt S w_1, U_S 
\big[|S|^{-1} S |S|^{-1}\big] (U_S)^* \hatt S w_2\big)_{\hatt \cH} \, ,    \lb{3.37}
\end{align}
where we used $|S|=(U_S)^* S$.
\end{proof}

Equation \eqref{3.20a} is of course motivated by rewriting the abstract linear pencil buckling eigenvalue problem \eqref{3.1c}, $S^* S u = \lambda S u$, $\lambda \neq 0$, in the form
\begin{equation}
\lambda^{-1} S^* S u = \lambda^{-1} (S^* S)^{1/2} \big[(S^* S)^{1/2} u\big]  
= S (S^* S)^{-1/2} \big[(S^* S)^{1/2} u\big]    \lb{3.38}
\end{equation}
and hence in the form of a standard eigenvalue problem
\begin{equation}
|S|^{-1} S |S|^{-1} w = \lambda^{-1} w, \quad \lambda \neq 0, \quad w = |S| u.  \lb{3.39}
\end{equation}

We conclude this section with a concrete example discussed explicitly in 
Grubb \cite{Gr83} (see also \cite{Gr68}--\cite{Gr71} for necessary background) and make the explicit connection with the buckling problem. It was this example which greatly motivated the abstract results in this note:

\begin{example} \lb{e3.6} $($\cite{Gr83}.$)$
Let $\cH = L^2(\Om; d^n x)$, with $\Om \subset\bbR^n$, $n \geq 2$, open and bounded, with a smooth boundary $\partial\Om$, and consider the minimal operator realization 
$S$ of the differential expression $\mathscr{S}$ in $L^{2}(\Om; d^n x)$, defined by 
\begin{align}
& S u = \mathscr{S} u,  \lb{3.40} \\
& u \in \dom(S) = H_0^{2m}(\Om) = \big\{v \in H^{2m}(\Om) \, \big| \, \gamma_{k} v =0, \, 
0 \leq k \leq 2m-1\big\},  \quad m \in\bbN,   \no
\end{align}
where 
\begin{align}
& \mathscr{S} = \sum_{0 \leq |\alpha| \leq 2m} a_{\alpha}(\cdot) D^{\alpha},   \lb{3.41} \\
& D^{\alpha} = (-i \partial/\partial x_1)^{\alpha_1} \cdots  
(-i\partial/\partial x_n)^{\alpha_n}, 
\quad \alpha =(\alpha_1,\dots,\alpha_n) \in \bbN_0^n,    \lb{3.42} \\
& a_{\alpha} (\cdot) \in C^\infty(\ol \Om),  \quad 
C^\infty(\ol \Om) = \bigcap_{k\in\bbN_0} C^k(\ol \Om),   \lb{3.43}
\end{align}
and the coefficients $a_\alpha$ are chosen such that $S$ is symmetric in 
$L^2(\bbR^n; d^n x)$, that is, the differential expression $\mathscr{S}$ is formally self-adjoint, 
\begin{equation}
(\mathscr{S} u, v)_{L^2(\bbR^n; d^n x)} = (u, \mathscr{S} v)_{L^2(\bbR^n; d^n x)}, \quad 
u, v \in C_0^\infty(\Om),   \lb{3.44}
\end{equation}
and $\mathscr{S}$ is strongly elliptic, that is, for some $c>0$, 
\begin{equation}
\Re\bigg(\sum_{|\alpha|=2m} a_{\alpha} (x) \xi^{\alpha}\bigg) \geq c |\xi|^{2m}, \quad 
x \in \ol \Om, \; \xi \in\bbR^n.   \lb{3.45}
\end{equation}
In addition, we assume that $S\geq \varepsilon I_{L^{2}(\Om; d^n x)}$ for some 
$\varepsilon >0$.\ The trace operators $\gamma_k$ are defined as follows: Consider 
\begin{equation}
\mathring{\gamma}_k : \begin{cases} C^\infty(\ol \Om) \to C^\infty(\partial \Om) \\
 u \mapsto (\partial^k_n u)|_{\partial\Om}, \end{cases}     \lb{3.46}
\end{equation}
with $\partial_n$ denoting the interior normal derivative. The map 
$\mathring{\gamma}$ then extends by continuity to a bounded operator 
\begin{equation}
\gamma_k : H^s(\Om) \to H^{s-k - (1/2)}(\partial\Om), \quad s > k + (1/2),  \lb{3.47}
\end{equation}
in addition, the map 
\begin{equation}
\gamma^{(r)} = (\gamma_0,\dots,\gamma_r) : H^s(\Om) \to 
\prod_{k=0}^r H^{s-k-(1/2)}(\partial\Om), \quad s > r +(1/2),    \lb{3.48}
\end{equation}
satisfies
\begin{equation}
\ker\big(\gamma^{(r)}\big) = H_0^s(\Om), \quad \ran\big(\gamma^{(r)}\big) 
= \prod_{k=0}^r H^{s-k-(1/2)}(\partial\Om).     \lb{3.49}
\end{equation}
Then $S^*$, the maximal operator realization of $\mathscr{S}$ in $L^{2}(\Om; d^n x)$, is given by
\begin{equation}
S^* u = \mathscr{S} u, \quad u \in \dom(S^*) 
= \big\{v\in L^{2}(\Om; d^n x) \,\big|\, \mathscr{S} v\in L^{2}(\Om; d^n x)\big\},  
\lb{3.50}
\end{equation}
and $S_F$ is characterized by
\begin{equation}
S_F u = \mathscr{S} u, \quad u \in \dom(S_F) 
= \big\{v \in H^{2m}(\Om) \, \big| \, \gamma_k v =0, \, 0 \leq k \leq m-1\big\}.  \lb{3.51}
\end{equation}
The Krein--von Neumann extension $S_K$ of $S$ then has the domain
\begin{equation}
\dom(S_K) = H_0^{2m}(\Om) \dotplus \ker(S^*), \quad \dim(\ker(S^*)) = \infty,   \lb{3.52}
\end{equation}
and elements $u \in \dom(S_K)$ satisfy the nonlocal boundary condition
\begin{align}
& \gamma_N u - P_{\gamma_D,\gamma_N} \gamma_D u =0,     \lb{3.53} \\
& \gamma_D u = (\gamma_0 u,\dots,\gamma_{m-1} u), \quad 
\gamma_N u = (\gamma_m u,\dots,\gamma_{2m-1} u), \quad u \in \dom(S_K),     
\lb{3.54}
\end{align}
where 
\begin{align}
\begin{split}
& P_{\gamma_D, \gamma_N} = \gamma_N \gamma_Z^{-1} : 
\prod_{k=0}^{m-1} H^{s-k-(1/2)} (\partial\Om) \to 
\prod_{j=m}^{2m-1} H^{s-j-(1/2)} (\partial\Om)    \\
& \quad \text{continuously for all $s\in\bbR$,}
\end{split}
\end{align}
and $\gamma_Z^{-1}$ denotes the inverse of the isomorphism $\gamma_Z$ given by
\begin{align}
&\gamma_D : Z_{\mathscr{S}}^s \to \prod_{k=0}^{m-1} H^{s-k-(1/2)} (\partial\Om),  \\
& Z_{\mathscr{S}}^s = \big\{u\in H^s(\Om) \, \big| \, \mathscr{S} u =0 \, \text{in  
$\Om$ in the sense of distributions in $\cD^\prime(\Om)$}\big\}, \quad s\in\bbR. 
\end{align}

Moreover one has
\begin{equation}
\big(\hatt S\big)^{-1} = \iota_{\cW} [I_{\cH} - P_{\gamma_D, \gamma_N} \gamma_D] 
\big(\hatt S_K\big)^{-1}, 
\end{equation}
since $[I_{\cH} - P_{\gamma_D, \gamma_N} \gamma_D] \dom(S_K) \subseteq \dom(S)$ 
and $S [I_{\cH} - P_{\gamma_D, \gamma_N} \gamma_D] v = \lambda v$, 
$v\in\dom(S_K)$.

As discussed in detail in Grubb \cite{Gr83},
\begin{equation}
\sigma_{\rm ess} (S_K) = \{0\}, \quad \sigma(S_K) \cap (0,\infty) = \sigma_{\rm d} (S_K) 
\lb{3.55}
\end{equation}
and the nonzero $($and hence discrete$)$ eigenvalues of $S_K$ satisfy a Weyl-type asymptotics. The connection to a higher-order buckling eigenvalue problem established by Grubb then reads
\begin{align}
& \text{There exists $0 \neq v \in S_K$ satisfying } \, \mathscr{S} v = \lambda v 
\, \text{ in } \, \Om, \quad 
\lambda \neq 0   \lb{3.57}  \\
& \text{if and only if}  \no \\
& \text{there exists $0 \neq u \in C^\infty (\ol \Om)$ such that } \, 
\begin{cases} \mathscr{S}^2 u = \lambda \, \mathscr{S} u \, \text{ in } \, \Om, 
\quad \lambda \neq 0,  \\  
\gamma_k u = 0, \; 0 \leq k \leq 2m-1,  \end{cases}  \lb{3.58}
\end{align}
where the solutions $v$ of \eqref{3.57} are in one-to-one correspondence with the solutions $u$ of \eqref{3.58} via
\begin{equation}
u = S_F^{-1} \mathscr{S} v, \quad v = \lambda^{-1} \mathscr{S} u.   \lb{3.59}  
\end{equation}
\end{example}

Since $S_F$ has purely discrete spectrum in Example \ref{e3.6}, we note that Theorem \ref{tKAS} applies in this case. 

\begin{remark} \lb{r3.7}
In the particular case $m=1$ and $\mathscr{S} = -\Delta$, the linear pencil eigenvalue problem \eqref{3.58} (i.e., the concrete analog of the abstract buckling eigenvalue problem $S^* S u = \lambda S u$, $\lambda \neq 0$, in \eqref{3.1c}), then yields the {\it buckling of a clamped plate problem},
\begin{equation}
(-\Delta)^2u=\lambda (-\Delta) u \,\text{ in } \,\Omega, \quad \lambda \neq 0, \; 
u\in H^2_0(\Omega),    \lb{3.60}
\end{equation}
as distributions in $H^{-2}(\Om)$. Here we used the fact that for any nonempty bounded open set $\Om\subset\bbR^n$, $n\in\bbN$, $n\geq 2$, 
$(-\Delta)^m\in \cB\big(H^k(\Om), H^{k-2m}(\Om)\big)$, $k\in\bbZ$, $m\in\bbN$. In addition, if $\Om$ is a Lipschitz domain, then one has that 
$-\Delta\colon H^1_0(\Om) \to H^{-1}(\Om)$ is an isomorphism and similarly, 
$(-\Delta)^2\colon H^2_0(\Om) \to H^{-2}(\Om)$ is an isomorphism. (For the natural norms on $H^k(\Om)$, $k\in\bbZ$, see, e.g., \cite[p.\ 73--75]{Mc00}.)
 We refer, for instance, to \cite[Sect.\ 4.3B]{Be77} for a derivation of \eqref{3.60} from the fourth-order system of quasilinear von K{\'a}rm{\'a}n partial differential equations. To be precise, \eqref{3.60} should also be considered in the special case $n=2$.
\end{remark}

\begin{remark} \lb{r3.8}
We emphasize that the smoothness hypotheses on $\partial\Om$ can be relaxed in the special case of the second-order Schr\"odinger operator associated with the differential expression $-\Delta + V$, where $V\in L^\infty(\Om; d^nx)$ is real-valued: Following the treatment of self-adjoint extensions of $S = \ol{(-\Delta + V)|_{C_0^\infty(\Om)}}$ on quasi-convex domains $\Om$ first introduced in \cite{GM09}, the case of the 
Krein--von Neumann extension $S_K$ of $S$ on such quasi-convex domains (which are close to minimally smooth) is treated in great detail in \cite{AGMT09}. In particular, 
a Weyl-type asymptotics of the associated (nonzero) eigenvalues of $S_K$ has been proven in \cite{AGMT09}. In the higher-order smooth case described in Example 
\ref{e3.6}, a Weyl-type asymptotics for the nonzero eigenvalues of $S_K$ has been proven by Grubb \cite{Gr83} in 1983.
\end{remark}

\noindent {\bf Acknowledgments.}
We are indebted to Gerd Grubb, Konstantin Makarov, Mark Malamud, 
and Eduard Tsekanovskii for a critical reading of our manuscript, and for providing us with numerous additional insights into this circle of ideas.


\end{document}